\documentclass[11pt,a4paper]{article}

\usepackage[utf8]{inputenc}
\usepackage[T1]{fontenc}
\usepackage{lmodern}
\usepackage{amsmath,amssymb,amsthm}
\usepackage{mathtools}
\usepackage[margin=1in]{geometry}
\usepackage{graphicx}
\usepackage{booktabs}
\usepackage{enumitem}
\usepackage{microtype}
\usepackage[colorlinks=true,linkcolor=blue,citecolor=blue,urlcolor=blue]{hyperref}

\theoremstyle{plain}
\newtheorem{theorem}{Theorem}
\newtheorem{proposition}{Proposition}
\newtheorem{lemma}{Lemma}

\theoremstyle{definition}
\newtheorem{problem}{Problem}
\newtheorem{remark}{Remark}
\newtheorem{construction}{Construction}

\newcommand{\chis}{\chi'_s}
\newcommand{\Kfour}{K_4}
\DeclareMathOperator{\col}{col}

\title{\bfseries A diamond-free claw-free cubic graph\\ with strong chromatic index 7}
\author{Kanishk Raj Tanwar\\[2pt]\normalsize \texttt{kanishkrajtanwar@gmail.com}}
\date{July 26, 2026}

\begin{document}
\maketitle

\begin{abstract}
A \emph{strong edge coloring} is a proper edge coloring in which every color class is an
induced matching; the least number of colors is the \emph{strong chromatic index}
$\chis(G)$. Lin and Lin proved that every claw-free subcubic graph other than the
triangular prism satisfies $\chis(G)\le 7$, with all their tight examples containing
diamonds. Kardo\v{s} (Problem~4.1 of the open-problem collection of the $33$rd Workshop on
Cycles and Colourings) asked whether every diamond-free claw-free cubic graph is strongly
$6$-edge-colorable, equivalently whether $\chis(T(G))=6$ for every cubic graph $G$, where
$T(G)$ is the truncation of $G$. We exhibit an explicit connected, simple, diamond-free,
claw-free cubic graph $H$ on $18$ vertices with $\chis(H)=7$, and show that it has the
fewest vertices possible for such a non-prism example. Hence the first formulation, over
simple cubic graphs, is false even after excluding the prism; and since $H$ is the
truncation of a cubic \emph{multigraph} with parallel edges, the two formulations are not
equivalent unless ``cubic graph'' is allowed to mean loopless multigraph, under which
reading the problem is answered negatively. The narrower version restricted to truncations
of \emph{simple} base graphs remains open. The graph $H$ was found with the assistance of
OpenAI's GPT-5.6 Sol Pro, and all claims are verified independently, both by hand and by
exact computation.
\end{abstract}



\section{Introduction}

Let $G$ be a simple graph. A \emph{strong edge coloring} assigns colors to the edges of
$G$ so that any two edges that are either adjacent or joined by an edge receive distinct
colors; equivalently, it is a proper vertex coloring of the square $L(G)^2$ of the line
graph of $G$. The minimum number of colors in a strong edge coloring is the
\emph{strong chromatic index} $\chis(G)=\chi\!\left(L(G)^2\right)$.

A graph is \emph{claw-free} if it contains no induced $K_{1,3}$, and \emph{diamond-free}
if it contains no $\Kfour-e$ as a subgraph (equivalently, no edge lies in two triangles).
Lv, Li and Zhang~\cite{LvLiZhang2022} proved that every claw-free subcubic graph other
than the triangular prism $C_3\,\square\,K_2$ satisfies $\chis(G)\le 8$. Lin and
Lin~\cite{LinLin2023} sharpened this to $\chis(G)\le 7$ and exhibited an infinite family
attaining the bound; crucially, every graph in that family contains a diamond.

Diamond-free claw-free cubic graphs are exactly the \emph{truncations} of cubic
(multi)graphs: the truncation $T(G)$ replaces each vertex of $G$ by a triangle, one corner
per incident edge~\cite{Kardos2025}. Since the diamond is the only obstruction removed by
truncation, it is natural to ask whether forbidding diamonds forces the strong chromatic
index down from $7$ to $6$. Kardo\v{s} posed this precisely.

\begin{problem}[Kardo\v{s} {\cite[Problem 4.1]{Kardos2025}}]\label{prob:41}
Is it true that every diamond-free claw-free cubic graph is strongly $6$-edge-colorable?
In other words, is it true that $\chis(T(G))=6$ for every cubic graph $G$?
\end{problem}

The only progress reported is by Han and Cui~\cite{HanCui2023}, who proved that the
truncated prisms are strongly $6$-edge-colorable. In this note we show that
Problem~\ref{prob:41}, taken over simple cubic graphs in its first formulation, is false,
and we analyze precisely which reading of the problem survives.\footnote{The graph $H$ was
discovered with the assistance of OpenAI's GPT-5.6 Sol Pro. All statements in this paper were
subsequently verified independently by the author, both by the self-contained argument of
Section~3 and by the exact computation described in Section~5.}

\begin{theorem}\label{thm:main}
There exists a connected, simple, diamond-free, claw-free cubic graph $H$ on $18$ vertices
that is not the triangular prism and satisfies $\chis(H)=7$.
\end{theorem}

Our example is small, explicit, and its strong chromatic index is established by a
self-contained forcing argument that we also confirm by an exhaustive computation of
$\chi\!\left(L(H)^2\right)$. Two features are worth emphasizing. First, the graph $H$ is
diamond-free by construction, so it is genuinely of the type the problem forbids diamonds
from. Second, the obstruction to a $6$-coloring is \emph{not} a clique: the square
$L(H)^2$ has clique number $6$ yet chromatic number $7$, so the difficulty is a true
$\chi>\omega$ phenomenon rather than a forced $K_7$.

\section{The construction}

\subsection{A nine-vertex gadget}

\begin{construction}\label{con:Q}
Let $Q$ be the graph on nine vertices built from three vertex-disjoint triangles
\[
  A = a_0a_1a_2a_0, \qquad B = b_0b_1b_2b_0, \qquad C = c_0c_1c_2c_0,
\]
together with the four connecting edges
\[
  a_0b_0, \qquad a_1c_0, \qquad b_1c_1, \qquad b_2c_2.
\]
\end{construction}

\begin{figure}[h]
\centering
\includegraphics[width=0.95\textwidth]{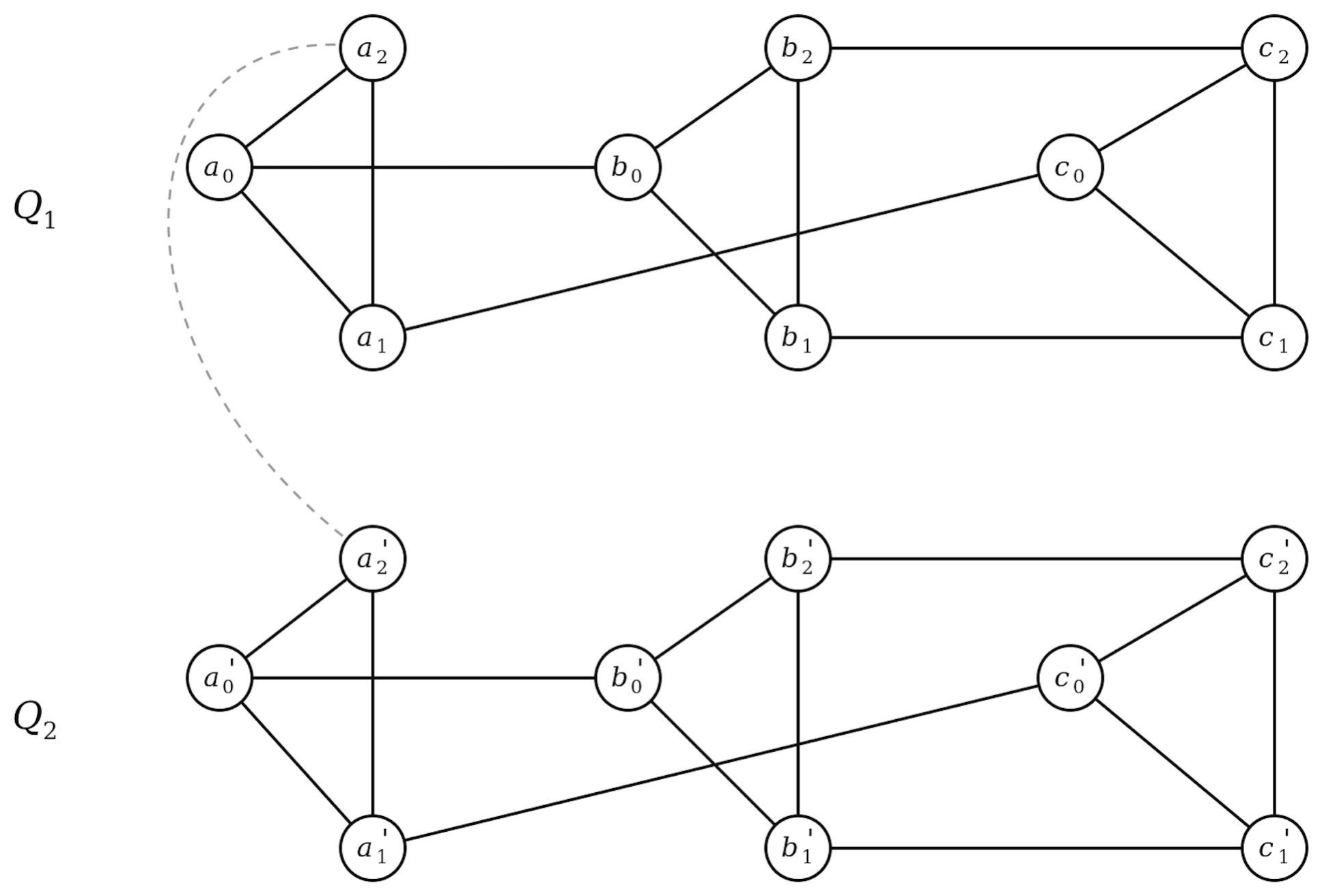}
\caption{The graph $H$: two copies $Q_1$ and $Q_2$ of the nine-vertex gadget $Q$, joined by
the bridge $a_2^{(1)}a_2^{(2)}$ (dashed). The vertices of the second copy $Q_2$ are marked
with primes. Within each copy the three triangles $A=a_0a_1a_2$, $B=b_0b_1b_2$,
$C=c_0c_1c_2$ are joined by the four connecting edges $a_0b_0$, $a_1c_0$, $b_1c_1$,
$b_2c_2$; the vertex $a_2$ is the unique degree-$2$ vertex of $Q$ and is raised to degree
$3$ by the bridge.}
\label{fig:Q}
\end{figure}

Every vertex of $Q$ has degree $3$ except $a_2$, which has degree $2$
(it lies only in triangle $A$ and receives no connecting edge). Thus $Q$ has $9$ vertices,
$3\cdot 3 + 4 = 13$ edges, and a single vertex of degree $2$.

\subsection{The cubic graph \texorpdfstring{$H$}{H}}

\begin{construction}\label{con:H}
Take two disjoint copies $Q_1,Q_2$ of $Q$, with corresponding degree-$2$ vertices
$a_2^{(1)}$ and $a_2^{(2)}$, and add the single bridge $a_2^{(1)}a_2^{(2)}$. Call the
resulting graph $H$.
\end{construction}

The bridge raises both degree-$2$ vertices to degree $3$, so
\[
  |V(H)| = 18, \qquad |E(H)| = 2\cdot 13 + 1 = 27,
\]
and $H$ is cubic. It is clearly simple and connected.

\begin{proposition}\label{prop:props}
$H$ is a connected, simple, diamond-free, claw-free cubic graph on $18$ vertices, and it is
not the triangular prism.
\end{proposition}

\begin{proof}
\emph{Claw-free.} Every vertex of $H$ lies in exactly one of the six triangles. Hence among
the three neighbors of any vertex $v$, the two that lie in $v$'s triangle are adjacent, so
$v$ cannot be the center of an induced $K_{1,3}$. Therefore $H$ is claw-free.

\emph{Diamond-free.} A diamond $\Kfour-e$ has a pair of adjacent vertices with two common
neighbors, so it suffices to show that no edge of $H$ has endpoints with two common
neighbors. If an edge lies inside one of the six triangles, its endpoints have exactly one
common neighbor, namely the third vertex of that triangle. If an edge joins two different
triangles (a connecting edge or the bridge), its endpoints lie in disjoint triangles and
have no common neighbor. Hence no edge has two common neighbors, and $H$ contains no
diamond, even as a non-induced subgraph.

The triangular prism has $6$ vertices, whereas $|V(H)|=18$.
\end{proof}

\section{Six colors are impossible}

Because $Q_1$ is a subgraph of $H$, it suffices to prove that $Q$ cannot be strongly
edge-colored with six colors; then $\chis(H)\ge\chis(Q)\ge 7$. Throughout this section we
work inside one copy $Q$ and suppress the copy superscript. We name the thirteen edges of
$Q$:
\[
  x=a_0b_0,
\]
\[
  u=b_0b_1,\quad v=b_0b_2,\quad w=b_1b_2,\qquad y=b_1c_1,\quad z=b_2c_2,
\]
\[
  r=c_1c_2,\quad s=c_0c_1,\quad t=c_0c_2,\qquad q=a_1c_0,\quad m=a_0a_1,\quad n=a_0a_2,\quad p=a_1a_2.
\]

\subsection{A forced six-edge clique}

\begin{lemma}\label{lem:clique}
The six edges $K=\{x,u,v,w,y,z\}$ are pairwise in conflict; that is, they form a $K_6$ in
$L(Q)^2$.
\end{lemma}

\begin{proof}
The edges $u,v,w$ are the three edges of triangle $B$, hence pairwise adjacent. The edges
$x,y,z$ leave the three vertices $b_0,b_1,b_2$ of $B$ respectively, so each is adjacent to
two of $u,v,w$, and any two of $x,y,z$ have endpoints joined by an edge of $B$. Thus every
pair among the six edges is either adjacent or joined by an edge of $B$.
\end{proof}

In any strong coloring with at most six colors, the edges of $K$ therefore receive six
distinct colors. We rename the colors $x,u,v,w,y,z$ so that each color is named after the
clique edge carrying it.

For each of the remaining seven edges, we list the colors of $K$ that are \emph{not}
forbidden by the clique --- that is, the colors of the clique edges at strong distance
greater than two from that edge:

\begin{center}
\begin{tabular}{cl}
\toprule
Edge & Colors not forbidden by $K$ \\
\midrule
$r$ & $\{x\}$ \\
$s$ & $\{x,v\}$ \\
$t$ & $\{x,u\}$ \\
$q$ & $\{u,v,w\}$ \\
$m$ & $\{w,y,z\}$ \\
$n$ & $\{w,y,z\}$ \\
$p$ & $\{u,v,w,y,z\}$ \\
\bottomrule
\end{tabular}
\end{center}

For instance $r=c_1c_2$ conflicts with each of $u,v,w,y,z$ but not with $x=a_0b_0$, so $r$
can only receive color $x$. (Each row is verified directly from the adjacencies of $Q$;
see also Section~\ref{sec:comp}.)

\subsection{The forcing chain}

\begin{proposition}\label{prop:noSix}
$\chis(Q)\ge 7$.
\end{proposition}

\begin{proof}
Suppose, for contradiction, that $Q$ has a strong $6$-edge-coloring. By
Lemma~\ref{lem:clique} the clique $K$ uses all six colors, named as above. We follow the
table.

\emph{Step 1.} The edge $r$ has the single available color $x$, so $\col(r)=x$. The edges
$s$ and $t$ each conflict with $r$, hence cannot take color $x$; their lists then force
$\col(s)=v$ and $\col(t)=u$.

\emph{Step 2.} The edge $q$ has list $\{u,v,w\}$ and conflicts with both $s$ (color $v$)
and $t$ (color $u$), so $\col(q)=w$.

\emph{Step 3.} The edges $m$ and $n$ have list $\{w,y,z\}$ and both conflict with $q$
(color $w$), leaving $\{y,z\}$. Since $m$ and $n$ conflict with each other, they use the
two colors $y$ and $z$ in some order:
\[
  \{\col(m),\col(n)\}=\{y,z\}.
\]

\emph{Contradiction.} The edge $p$ conflicts with all five edges $s,t,q,m,n$, which now
carry precisely the five colors $v,u,w,y,z$. Its list is $\{u,v,w,y,z\}$, and color $x$ is
already forbidden because $p$ conflicts with the clique edge $x=a_0b_0$. Hence $p$ has no
available color, a contradiction. Therefore no strong $6$-edge-coloring exists and
$\chis(Q)\ge 7$.
\end{proof}

\begin{proof}[Proof of Theorem~\ref{thm:main}]
By Proposition~\ref{prop:props}, $H$ is a connected, simple, diamond-free, claw-free cubic
graph on $18$ vertices, not the triangular prism. Since $Q\cong Q_1\subseteq H$, strong
colorings restrict to subgraphs, so $\chis(H)\ge\chis(Q)\ge 7$ by
Proposition~\ref{prop:noSix}. Conversely $H$ is claw-free, subcubic, and not the triangular
prism, so Lin and Lin's bound~\cite{LinLin2023} gives $\chis(H)\le 7$. Therefore
$\chis(H)=7$.
\end{proof}

\section{The two formulations are not equivalent}

Problem~\ref{prob:41} states its two forms as ``in other words'' equivalent. They are not,
unless ``cubic graph'' is read to include loopless multigraphs.

\begin{proposition}\label{prop:multigraph}
Let $M$ be the graph obtained from $H$ by contracting each of its six triangles to a single
vertex. Then $M$ is the loopless cubic \emph{multigraph} on the six vertices
$A_1,B_1,C_1,A_2,B_2,C_2$ whose edges are, for each $i\in\{1,2\}$,
\[
  A_iB_i,\qquad A_iC_i,\qquad \text{a double edge } B_iC_i,
\]
together with the edge $A_1A_2$; and $H=T(M)$.
\end{proposition}

\begin{proof}
In a diamond-free claw-free cubic graph the triangles are vertex-disjoint and every vertex
lies in exactly one triangle, so contracting the triangles is well defined and inverts the
truncation, giving $H=T(M)$. Reading off the connecting edges of each copy $Q_i$: the edge
$a_0b_0$ becomes $A_iB_i$; the edge $a_1c_0$ becomes $A_iC_i$; the two edges $b_1c_1$ and
$b_2c_2$ both join $B$ to $C$, becoming a double edge $B_iC_i$; and the bridge
$a_2^{(1)}a_2^{(2)}$ becomes $A_1A_2$. Every vertex of $M$ has degree $3$. Moreover this
triangle decomposition is unique --- as noted, every vertex of a diamond-free claw-free
cubic graph lies in exactly one triangle --- so $M$ is the only (multi)graph with
$H=T(M)$; in particular $H$ is the truncation of no \emph{simple} cubic graph.
\end{proof}

\begin{figure}[h]
\centering
\includegraphics[width=0.55\textwidth]{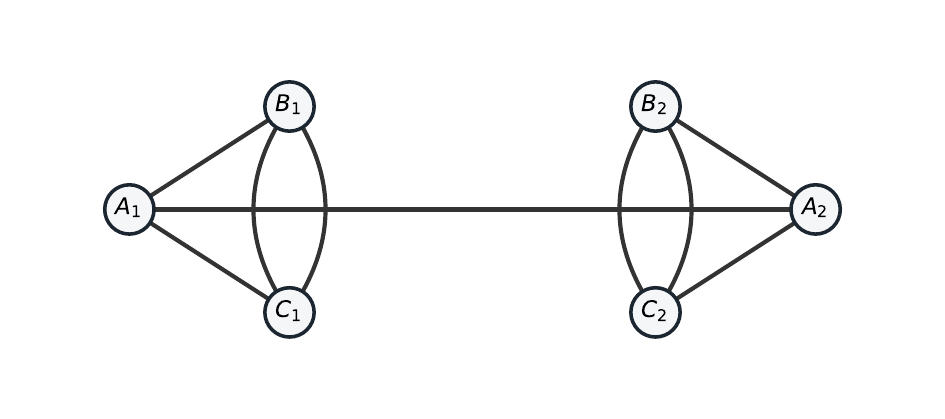}
\caption{The cubic multigraph $M$ with $H=T(M)$. Each pair $B_iC_i$ is joined by a double
edge; these parallel edges are exactly why $M$ is not simple.}
\label{fig:M}
\end{figure}

The consequences are as follows.
\begin{itemize}[itemsep=2pt]
  \item If ``cubic graph'' in the second formulation admits loopless multigraphs, then that
  formulation is also false, since $\chis(T(M))=\chis(H)=7$.
  \item If ``cubic graph'' means \emph{simple} cubic graph, then $M$ is excluded because it
  has parallel edges. In that reading the second formulation is strictly narrower than the
  first, and $H$ does not refute it.
  \item The triangular prism has the same character: it is $T(M_0)$ for $M_0$ the two-vertex
  cubic multigraph with three parallel edges, and it satisfies $\chis=9$.
\end{itemize}

Thus the genuinely narrower question left open is
\[
  \boxed{\ \chis(T(G))=6 \quad\text{for every \emph{simple} cubic graph } G.\ }
\]
Our counterexample does not settle this restricted version.

\begin{remark}
The first formulation of Problem~\ref{prob:41}, read literally over simple graphs, already
fails on the triangular prism, for which $L(P)^2=K_9$ and hence $\chis(P)=9$; this is
recorded by Lin and Lin~\cite{LinLin2023}. That single exception could be dismissed as an
accidentally omitted base case. Theorem~\ref{thm:main} shows something stronger: excluding
the prism does not repair the statement, because $H$ is an $18$-vertex counterexample far
from the prism. As a sanity check on the boundary of the phenomenon, the Tutte $8$-cage
(the Tutte--Coxeter graph, a cubic graph of girth $8$ on $30$ vertices) is \emph{not} a
truncation and satisfies $\chis=7$; we verified this by the same computation, which also
shows that its conflict graph has clique number only $5$. Thus $6$ is already too few for
cubic graphs in general, so the truncation structure is doing real work in whatever form of
the conjecture survives.
\end{remark}

\section{Computational verification}\label{sec:comp}

Every claim above was checked by exact computation on the explicit graph. Building $H$ from
Constructions~\ref{con:Q}--\ref{con:H} and forming the conflict graph
$L(H)^2$ (two edges adjacent iff they share an endpoint or are joined by an edge of $H$)
yields a graph on $27$ vertices and $122$ edges. A branch-and-bound chromatic-number
computation gives
\[
  \chi\!\left(L(H)^2\right)=7, \qquad \omega\!\left(L(H)^2\right)=6,
\]
confirming $\chis(H)=7$ and, in particular, that the obstruction is not a $K_7$. The same
routine gives $\chi\!\left(L(Q)^2\right)=7$ with a maximum clique of size $6$ realized by
$K=\{x,u,v,w,y,z\}$, and an independent check reproduces every row of the availability table
of Section~3 and every conflict used in the forcing chain. An explicit strong
$7$-edge-coloring of $H$, which establishes the upper bound $\chis(H)\le 7$ directly and
independently of~\cite{LinLin2023}, is recorded in Appendix~\ref{app:coloring}. The
verification script that reproduces every computation in this paper --- including the
minimality enumeration below --- is provided as an ancillary file (\texttt{anc/verify.py}).

\subsection{Minimality}

The same enumeration shows that $H$ is as small as such a graph can be.

\begin{proposition}\label{prop:min}
Every connected diamond-free claw-free cubic graph on fewer than $18$ vertices, other than
the triangular prism, is strongly $6$-edge-colorable. Hence $H$ has the minimum number of
vertices among connected diamond-free claw-free cubic graphs, other than the triangular
prism, with strong chromatic index $7$.
\end{proposition}

\begin{proof}
By Proposition~\ref{prop:multigraph} every diamond-free claw-free cubic graph is the
truncation $T(M)$ of a loopless cubic multigraph $M$, and $|V(T(M))|=3\,|V(M)|$. A cubic
multigraph has an even number of vertices, so $|V(M)|$ is even and $|V(T(M))|$ is a multiple
of $6$; the only orders below $18$ are therefore $6$ and $12$. On $6$ vertices the sole
connected such graph is the triangular prism, with $\chis=9$. For $12$ vertices there are,
up to isomorphism, exactly two connected loopless cubic multigraphs on four vertices --- the
complete graph $K_4$, whose truncation $T(K_4)$ is the truncated tetrahedron, and the
$4$-cycle with two opposite edges doubled --- and for each we computed
$\chi\!\left(L(T(M))^2\right)=6$. Thus no non-prism counterexample exists below $18$
vertices, and $H$ attains $18$.
\end{proof}

\appendix

\section{An explicit strong \texorpdfstring{$7$}{7}-edge-coloring of \texorpdfstring{$H$}{H}}
\label{app:coloring}

The table below assigns one of the colors $1,\dots,7$ to every edge of $H$, using the edge
names of Section~3 within each copy; the bridge $a_2^{(1)}a_2^{(2)}$ receives color $3$.
One checks directly (or by the routine of Section~\ref{sec:comp}) that any two edges that
are adjacent or joined by an edge receive different colors, so this is a strong
$7$-edge-coloring and $\chis(H)\le 7$. Together with $\chis(H)\ge 7$ from
Proposition~\ref{prop:noSix} this gives $\chis(H)=7$ without appeal to~\cite{LinLin2023}.

\begin{center}
\begin{tabular}{l ccccc ccccc ccc}
\toprule
& $x$ & $u$ & $v$ & $w$ & $y$ & $z$ & $r$ & $s$ & $t$ & $q$ & $m$ & $n$ & $p$ \\
\midrule
Copy $Q_1$ & 1 & 2 & 3 & 6 & 4 & 5 & 7 & 3 & 1 & 2 & 4 & 5 & 6 \\
Copy $Q_2$ & 1 & 2 & 3 & 6 & 4 & 5 & 7 & 3 & 1 & 2 & 5 & 4 & 7 \\
\bottomrule
\end{tabular}
\end{center}

\noindent Here $x=a_0b_0$, $u=b_0b_1$, $v=b_0b_2$, $w=b_1b_2$, $y=b_1c_1$, $z=b_2c_2$,
$r=c_1c_2$, $s=c_0c_1$, $t=c_0c_2$, $q=a_1c_0$, $m=a_0a_1$, $n=a_0a_2$, $p=a_1a_2$ within
each copy, and the bridge $a_2^{(1)}a_2^{(2)}$ has color $3$.

\section*{Acknowledgments}

The counterexample $H$ was found with the assistance of OpenAI's GPT-5.6 Sol Pro; the model
proposed the gadget construction, which the author then checked and formalized. Every claim
was verified independently, both by hand and by the exact computation of
Section~\ref{sec:comp}. The author also thanks Franti\v{s}ek Kardo\v{s} for posing
Problem~4.1, whose precise phrasing motivated the distinction between the simple and
multigraph readings analyzed here.

\end{document}